\documentclass[11pt]{article}
\usepackage{}
\usepackage{amscd}
\usepackage{amsmath}
\usepackage{latexsym}
\usepackage{amsfonts}
\usepackage{amssymb}
\usepackage{amsthm}
\usepackage{graphicx}
\usepackage{verbatim}
\usepackage{enumerate}
\usepackage{xcolor}

\theoremstyle{plain}\newtheorem{definition}{Definition}[section]
\theoremstyle{definition}\newtheorem{theorem}{Theorem}[section]
\theoremstyle{plain}\newtheorem{lemma}[theorem]{Lemma}
\theoremstyle{plain}
\theoremstyle{plain}\newtheorem{proposition}[theorem]{Proposition}
\theoremstyle{remark}\newtheorem{remark}{Remark}[section]

\newcommand{\Div}{\mathrm{div}}
\newcommand{\R}{\mathbb{R}}
\newcommand{\be}{\begin{equation}}
\newcommand{\ee}{\end{equation}}
\newcommand{\ba}{\begin{aligned}}
\newcommand{\ea}{\end{aligned}}
\newcommand{\f}{\frac}

\newcommand{\ben}{\begin{enumerate}}
\newcommand{\een}{\end{enumerate}}

\newcommand{\Rmnum}[1]{\expandafter\@slowromancap\romannumeral #1@}
\allowdisplaybreaks

\begin{document}
\title{Decay of solutions to the three-dimensional generalized Navier-Stokes equations}
\author{}
\date{}
\maketitle \maketitle \centerline{\scshape Quansen Jiu\footnote{The
research is partially supported by National Natural Sciences
Foundation of China (No. 11171229, 11231006 and 11228102) and
Project of Beijing Chang Cheng Xue Zhe.}}

{\centerline{School of Mathematical Sciences, Capital Normal
University}
  \centerline{Beijing  100048, P. R. China}
   \centerline{  \it Email: jiuqs@mail.cnu.edu.cn}

\vspace{3mm}
\centerline{\scshape Huan Yu}
\centerline{School of Mathematical Sciences, Capital Normal University}
\centerline{Beijing  100048, P. R. China}
\centerline{\it Email: yuhuandreamer@163.com}

\begin{abstract}
In this paper, we first obtain the temporal decay estimates for weak
solutions to the three  dimensional generalized Navier-Stokes
equations. Then, with these estimates at disposal, we obtain the
temporal decay estimates for higher order derivatives of the smooth
solution with small initial data.   The decay rates are optimal in
the sense that they coincides with ones of the corresponding
generalized heat equation. These results improve the previous known
results  to the classical Navier-Stokes equations.
\end{abstract}

\section{Introduction}
\setcounter{section}{1}\setcounter{equation}{0}
The incompressible Navier-Stokes equations can be written as
\be\label{eq-0}
\left\{\ba
&u_{t}+(u\cdot\nabla)u-\nu\Delta u=-\nabla p,\\
&\Div~ u=0,\\
&u(x,0)=u_{0}(x),\\  \ea\ \right. \ee
where $x\in
\mathbb{R}^n,~n\geq2,~ t>0$, the vector field $u=u(x,t)$ denotes the velocity of the fluid, $p=p(x,t)$ is the pressure of the fluid and the positive $\nu$ is the viscosity coefficient.

Whether or not weak solutions of \eqref{eq-0} decay to zero in
$L^{2}$ as  time tends to infinity was posed by Leray in his
pioneering paper \cite{[Leray1],[Leray2]}. Kato \cite{[Kato]} gave
the first affirmative answer to the strong solutions with small data
to system \eqref{eq-0}. Algebraic decay rates for weak solutions to
system \eqref{eq-0} were first obtained by Schonbek
\cite{[Schonbek1]}, in which the Fourier splitting method was
introduced to prove that there exists a Leray-Hopf weak solution of
\eqref{eq-0} in three space dimension with arbitrary data in
$L^{1}\cap L^{2}$, satisfying
$$ \|u(t)\|_{2}\leq C(t+1)^{-\frac{1}{4}}$$ where  the constant C  depends only  on the $L^{1}$ and $L^{2}$
norms of the initial data. Later the method in \cite{[Schonbek1]}
was extended by Schonbek \cite{[Schonbek2]} (see also Kajikiya and
Miyakawa \cite{[K-M]},  Wiegner \cite{[W]} for the case $\R^{n}$
(n=2,3,4)) and it was proved that the decay rate  for Leray-Hopf
solutions of \eqref{eq-0} in three space dimension with large data
in $L^{p}\cap L^{2}$ with $1\leq p<2$ is same as those for the
solution of the heat equation. That is,
$$ \|u(t)\|_{2}\leq C(t+1)^{-\frac{3}{4}(\frac{2}{p}-1)},$$ where  the constant C only depends on the $L^{p}$ and $L^{2}$ norms of the initial data.
On the decay of solutions to the Navier-Stokes equations, it is also
referred to \cite{[B-M],[C],[H],[K-O],[Mare],[U]} and references
therein.

In this paper, we consider the large-time behavior of solutions to
the following Cauchy problem for the incompressible generalized
Navier-Stokes equations  \be\label{eq-1} \left\{\ba
&u_{t}+(u\cdot\nabla)u+\nu\Lambda^{2\alpha}u=-\nabla p,\\
&\Div~ u=0,\\
&u(x,0)=u_{0}(x), \ea\ \right. \ee where $x\in
\mathbb{R}^n,~n\geq2,~ t>0$,
 $\Lambda^{2\alpha}$ is
defined through Fourier transform (see \cite{[STEIN]})
$$~~\widehat{\Lambda^{2\alpha}f}(\xi)=|\xi|^{2\alpha}\widehat{f}(\xi),
~~\widehat{f}(\xi)=\int_{\R^{n}}f(x)e^{-2\pi ix\cdot\xi}dx.$$ It is
known that if $(u(x,t), p(x,t))$ is a solution to the
three-dimensional generalized Navier-Stokes equations, then for any
$\lambda>0$, the scalings $(u_{\lambda}(x,t),p_{\lambda}(x,t)) =
(\lambda^{2\alpha-1}u(\lambda x,\lambda^{2\alpha}t),
\lambda^{4\alpha-2}p(\lambda x,\lambda^{2\alpha}t))$ also solves the
generalized Navier-Stokes equations. The corresponding energy is
$$E(u_{\lambda})=\displaystyle\sup_{t}\int_{\R^{3}}|u_{\lambda}|^{2} dx+\int^{\infty}_{0}\int
_{\R^{3}}|\Lambda^{\alpha}u_{\lambda}|^{2} dxdt
=\lambda^{4\alpha-5}E(u).$$  It follows that
$E(u_{\lambda})\rightarrow \infty$ as $\lambda\rightarrow\infty$
when $\alpha<\frac{5}{4}$. In this sense, we say that the
three-dimensional generalized Navier-Stokes equations \eqref{eq-1}
is supercritical if $\alpha<\frac{5}{4}$, critical for
$\alpha=\frac{5}{4}$ and subcritical with $\alpha>\frac{5}{4}$. It
has been proved that when $\alpha\geq \f{5}{4}$, the
three-dimensional generalized Navier-Stokes equations admits a
global and unique regular solution (see \cite{[L]},
\cite{[Wujiahong]} for instance).

In this paper, we are concerned with the asymptotic behavior of
solution of \eqref{eq-1} in the supercritical case
$\alpha<\frac{5}{4}$. Motivated by
\cite{[Schonbek1]}-\cite{[Schonbek3]}, we will show that the weak
solutions to \eqref{eq-1}  subject to large initial data decay in
$L^{2}$ at a uniform algebraic rate. The decay estimates for the
higher order derivatives of the smooth solution with small initial
data will also be established in $L^{2}$.  To prove our main
results, the Fourier splitting method due to Schonbek
\cite{[Schonbek1]} with appropriate modification will be applied. It
should be noted that the decay rates obtained in this paper are
optimal in the sense that they coincide with ones of the
corresponding generalized heat equation
$v_{t}+\nu\Lambda^{2\alpha}v=0$ with the same initial data $u_{0}$
(see Lemma 3.1 in \cite{[M]}). Therefore, our results improve ones
obtained
 in \cite{[Schonbek2]} in which the classical Navier-Stokes
 equations ($\alpha=1$ in \eqref{eq-1}) are investigated. For
 completeness, the proof of existence of weak
solutions will be sketched
 in Appendix  in the end of the paper.

Throughout the rest of the paper the $L^{p}$- norm of a function $f$ is denoted by $\|f\|_{p}$ and the $H^{s}$- norm by $\|f\|_{H^{s}}$. We will also set $\nu=1$
for simplicity.

Our main results are listed as follows.

\begin{theorem}\label{the0}
 Let $0<\alpha\leq1$. Then for divergence-free vector-field $u_{0}\in L^{2}(\R^{3})\cap L^{p}(\R^{3})$ with $1\leq p<2$,  the system \eqref{eq-1} admits a weak solution such that
\begin{equation}\label{1}
\|u(t)\|_{2}^{2}\leq C(t+1)^{-\frac{3}{2\alpha}(\frac{2}{p}-1)},
\end{equation}  where the constant C depends on $\alpha$, the $L^{p}$ and $L^{2}$ norms of the initial data.
\end{theorem}

\begin{theorem}\label{the1}
 Let $1\leq\alpha<\frac{5}{4}$. Then for divergence-free vector-field $u_{0}\in L^{2}(\R^{3})\cap L^{p}(\R^{3})$ with $\frac{1}{3-2\alpha}\leq p<2$,  the system \eqref{eq-1} admits a  weak solution such that
\begin{equation}\label{2}\|u(t)\|_{2}^{2}\leq C(t+1)^{-\frac{3}{2\alpha}(\frac{2}{p}-1)},
\end{equation} where the constant C depends on $\alpha$, the $L^{p}$ and $L^{2}$ norms of the initial data.
\end{theorem}

The following are decay estimates for the higher order derivatives
of the smooth solution, of which global-in-time existence for
sufficiently small initial data is guaranteed  in \cite{[Wu]}.

\begin{theorem}\label{the2}
Let $0<\alpha\leq 1$ and $u_{0}\in L^{2}(\R^{3})\cap L^{1}(\R^{3})$
with $\Div~ u_{0}=0$.
 Then,  for $m\in \mathbb{N}$ (the set of positive integers), there exist  $T_{0}>0$ and $C>0$ such  that the small
 global-in-time solution satisfies  $$\|D^{m}u(t)\|_{2}^{2}\leq C(t+1)^{-\frac{3}{2\alpha}-\frac{m}{\alpha}}$$
 for all $t>T_{0}$, where the constant $C$ depends on $m$, $\alpha$ and $\|u_{0}\|_{L^{2}\cap L^{1}}$.
\end{theorem}

\vspace{3mm}
\begin{remark} The following cases can be dealt with in a similar
fashion:

(1)  If  $0<\alpha\leq \frac{1}{2}$ and $u_{0}\in L^{2}(\R^{3})\cap
L^{p}(\R^{3})$ with $1\leq p\leq\frac{6}{4\alpha+3}$, one has
$$\|D^{m}u(t)\|_{2}^{2}\leq
C(t+1)^{-\frac{3}{2\alpha}(\frac{2}{p}-1)-\frac{m}{\alpha}}.$$ To
prove this result, we just modify the estimate \eqref{k3} as
$$\|\nabla u\|_{\infty}\leq
C(t+1)^{-\frac{3}{4\alpha}(\frac{2}{p}-1)}.$$

(2) If $\frac{1}{2}<\alpha\leq 1$ and $u_{0}\in L^{2}(\R^{3})\cap
L^{p}(\R^{3})$ with $1\leq p\leq\frac{6}{4\alpha+1}$, one has
$$\|D^{m}u(t)\|_{2}^{2}\leq
C(t+1)^{-\frac{3}{2\alpha}(\frac{2}{p}-1)-\frac{m}{\alpha}}.$$ To
prove this result, we just modify the estimate
 \eqref{h} as
$$\|u\|_{\infty}\leq C(t+1)^{-\frac{3}{4\alpha}(\frac{2}{p}-1)}.$$
\end{remark}

\begin{theorem}\label{addthe}
Let $1\leq\alpha< \frac{5}{4}$ and $u_{0}\in L^{2}(\R^{3})\cap
L^{p}(\R^{3})$  with $\Div~ u_{0}=0$ and $\frac{1}{3-2\alpha}\leq
p<2$.
 Then,  for $m\in \mathbb{N}$ (the set of positive integers),  there exist  $T_{0}>0$ and $C>0$ such that the
 small global-in-time solution satisfies
 $$\|D^{m}u(t)\|_{2}^{2}\leq C(t+1)^{-\frac{3}{2\alpha}(\frac{2}{p}-1)-\frac{m}{\alpha}}$$ for all $t>T_{0}$,
 where the constant $C$ depends on $m$, $\alpha$ and  $\|u_{0}\|_{L^{2}\cap L^{p}}$.
\end{theorem}

\vspace{3mm}
\begin{remark}
The decay rates for higher order of derivatives of the solutions was
studied in \cite{[C-S]} for the classical Navier-Stokes equations
and in \cite{[Schonbek3]} for the Hall-magnetohydrodynamic
equations.
\end{remark}

The paper unfolds as follows: Section 2 is devoted to the proof of
Theorem \ref{the0} and Theorem \ref{the1} whereas Section 3 deals
with the proof of Theorem \ref{the2} and  Theorem \ref{addthe}. The
existence of weak solutions  is given  in the Appendix in the end of
 the paper.

\section{Proof of Theorem \ref{the0} and Theorem \ref{the1}}
\setcounter{section}{2}\setcounter{equation}{0}
In this section, Theorem \ref{the0} and Theorem \ref{the1} will be  proved.  We start with two key lemmas.

\begin{lemma}\label{lemb1}
Let $u$ be a smooth solution to system \eqref{eq-1} with initial data $u_{0}\in L^{p}(\R^{3})\cap L^{2}(\R^{3})$, $1\leq p<2$. Then there exists a  constant $C>0$ depending only on $\|u_{0}\|_{2}$ such that \begin{equation}\label{add4}
|\widehat{u}(\xi,t)|\leq C(|\widehat{u_{0}}(\xi)|+\frac{1}{|\xi|^{2\alpha-1}}).
\end{equation}
\end{lemma}

\begin{proof}
Taking the Fourier transform of the first equation of \eqref{eq-1} yields
\begin{equation}\label{b6}
\widehat{u}_{t}(\xi,t)+|\xi|^{2\alpha}\widehat{u}(\xi,t)=H(\xi,t),
\end{equation}
where $$H(\xi,t)=-\widehat{u\cdot\nabla u}(\xi,t)-\widehat{\nabla
p}(\xi,t).$$ Multiplying \eqref{b6} by  $e^{|\xi|^{2\alpha}t}$ gives
$$\frac{d}{dt}(e^{|\xi|^{2\alpha}t}\widehat{u}(\xi,t))=e^{|\xi|^{2\alpha}t}H(\xi,t).$$
Integrating with respect to  time  from $0$ to $t$, we have
\begin{equation}\label{bd}
\widehat{u}(\xi,t)=e^{-|\xi|^{2\alpha}t}\widehat{u_{0}}(\xi)
+\int_{0}^{t}e^{-|\xi|^{2\alpha}(t-s)}H(\xi,s)ds.
\end{equation}
Hence
\begin{equation}\label{add0}
|\widehat{u}(\xi,t)|\leq |\widehat{u_{0}}(\xi)|+
\int_{0}^{t}e^{-|\xi|^{2\alpha}(t-s)}|H(\xi,s)|ds.
\end{equation}

To complete the proof we need to establish an estimate for $H(\xi,s)$.
Taking the divergence operator on the first equation of \eqref{eq-1} yields
$$-\Delta p=\displaystyle{\sum_{i,j=1}^{3}}\frac{\partial^{2}}{\partial x_{i}\partial x_{j}}(u^{i}u^{j}).$$
Since the Fourier transform is a bounded map from $L^{1}$ into
$L^{\infty}$, it follows that
\begin{equation*}
\begin{split}
|\widehat{\nabla p}(\xi,t)|\leq& |\xi||\widehat{p}(\xi,t)|\\ \leq&
\displaystyle{\sum_{i,j=1}^{3}}\frac{|\xi_{i}\xi_{j}|}{|\xi|}|\widehat{u^{i}u^{j}}(\xi,t)|
\\ \leq&C|\xi|\|u(t)u(t)\|_{1}\\ \leq&C|\xi|\|u(t)\|_{2}^{2}.
\end{split}
\end{equation*}
Similarly, for the convection term,  using the divergence free
condition,  we have
\begin{equation*}
\begin{split}
|\widehat{u\cdot\nabla u}(\xi,t)|\leq& \displaystyle{\sum_{i}^{3}}|\xi||\widehat{u^{i}u}(\xi,t)|\\ \leq&
C|\xi|\|u(t)u(t)\|_{1}\\ \leq&C|\xi|\|u(t)\|_{2}^{2}.
\end{split}
\end{equation*}
Combing the above two estimates, we obtain
\begin{equation}\label{add6}|H(\xi,t)|\leq C|\xi| \|u(t)\|_{2}^{2}.
\end{equation}
Inserting \eqref{add6} into \eqref{add0} and using the boundedness
of the $L^{2}$ norm of the solution lead to
\begin{equation*}
\begin{split}
|\widehat{u}(\xi,t)|\leq&|\widehat{u_{0}}(\xi)| +\frac{ C}{|\xi|^{2\alpha-1}}\|u_{0}\|_{2}^{2}(1-e^{-|\xi|^{2\alpha}t})\\ \leq&C(|\widehat{u_{0}}(\xi)|+\frac{1}{|\xi|^{2\alpha-1}}).
\end{split}
\end{equation*}
The proof of the lemma is finished.
\end{proof}

\begin{lemma}\label{addlemma}
Let $u_{0}\in L^{p}(\R^{3})$ with $1\leq p\leq2$. Then
\begin{equation}\label{add3}
\int_{S(t)}|\widehat{u_{0}}(\xi)|^{2}d\xi\leq C(t+1)^{-\frac{3}{2\alpha}(\frac{2}{p}-1)},
\end{equation} where
\begin{equation}\label{b3}
S(t)=\{\xi\in \R^{3}: |\xi|\leq g(t)\},~~
g(t)=(\frac{\gamma}{t+1})^{\frac{1}{2\alpha}},
\end{equation} the constant C depends on $\gamma$ and the $L^{p}$ norm of $u_{0}$.
\end{lemma}

\begin{proof}
Denote $\mathcal{F}$  the Fourier transform. By Riesz theorem, if
$1\leq p\leq2$, the Fourier transform $\mathcal{F}: L^{p}\rightarrow
L^{q}$ is bounded, and
\begin{equation}\label{add1}
\|\mathcal{F}u_{0}\|_{q}\leq C\|u_{0}\|_{p},
\end{equation}
where $\frac{1}{p}+\frac{1}{q}=1$. Consequently, one has
\begin{equation}\label{add2}
\int_{S(t)}|\widehat{u_{0}}|^{2}d\xi\leq (\int_{S(t)}|\widehat{u_{0}}|^{q}d\xi)^{\frac{2}{q}}(\int_{S(t)}d\xi)^{1-\frac{2}{q}}.
\end{equation}
Thanks to \eqref{add1} and noting that the volume
$|S(t)|=Cg^{3}(t)$, we get
$$\int_{S(t)}|\widehat{u_{0}}(\xi)|^{2}d\xi\leq C(t+1)^{-\frac{3}{2\alpha}(\frac{2}{p}-1)}.$$
The proof of the lemma is finished.

\end{proof}

In the rest of this section, we first present a formal argument by
the Fourier splitting method (see \cite{[Schonbek1]}).

\begin{proof}[\textbf{Proof of Theorem \ref{the0}}]
By taking $L^2$-inner product on both sides of the first equation of
\eqref{eq-1} with $u$, we get
\begin{equation*}
\frac{d}{dt}\|u(t)\|^{2}_{2}=-2\|\Lambda^{\alpha}u(t)\|^{2}_{2}.
\end{equation*}
Applying the Plancherel theorem, one has
\begin{equation*}
\frac{d}{dt}\int_{\R^{3}}|\widehat{u}(\xi)|^{2}d\xi=-2\int_{\R^{3}}|\xi|^{2\alpha}|\widehat{u}(\xi)|^{2}d\xi.
\end{equation*}
Let
\begin{equation}\label{b3}
S(t)=\{\xi\in \R^{3}: |\xi|\leq g(t)\},  ~~ g(t)=(\frac{\gamma}{t+1})^{\frac{1}{2\alpha}},
\end{equation} where $\gamma$ is a constant to be determined.
Then
\begin{equation}\label{b4}
\begin{split}
\frac{d}{dt}\int_{\R^{3}}|\widehat{u}(\xi)|^{2}d\xi
\leq&-g^{2\alpha}(t)\int_{|\xi|\geq g(t)}|\widehat{u}(\xi)|^{2}d\xi-\int_{|\xi|\leq g(t)}|\xi|^{2\alpha}|\widehat{u}(\xi)|^{2}d\xi\\ \leq&-g^{2\alpha}(t)\int_{\R^{3}}|\widehat{u}(\xi)|^{2}d\xi+g^{2\alpha}(t)\int_{|\xi|\leq g(t)}|\widehat{u}(\xi)|^{2}d\xi.
\end{split}
\end{equation}
Multiplying \eqref{b4} by
$G(t)=e^{\int_{0}^{t}g^{2\alpha}(\tau)d\tau}$ yields
\begin{equation*}
\frac{d}{dt}(G(t)\|u(t)\|^{2}_{2})\leq g^{2\alpha}(t)G(t)\int_{|\xi|\leq g(t)}|\widehat{u}(\xi)|^{2}d\xi.
\end{equation*}
Note that $G(t)=(t+1)^{\gamma}$ by \eqref{b3}. It follows that
\begin{equation}\label{b5}
\frac{d}{dt}((t+1)^{\gamma}\|u(t)\|^{2}_{2})\leq \gamma(t+1)^{\gamma-1}\int_{|\xi|\leq g(t)}|\widehat{u}(\xi)|^{2}d\xi.
\end{equation}

To complete the proof we will use  Lemma \ref{lemb1} and
\ref{addlemma} to estimate the right hand of \eqref{b5}. Indeed, by
plugging \eqref{add4} into the right hand of \eqref{b5} and using
\eqref{add3}, we have
\begin{equation*}
\begin{split}
\frac{d}{dt}((t+1)^{\gamma}\|u(t)\|^{2}_{2})\leq& C(t+1)^{\gamma-1}\int_{|\xi|\leq g(t)}|\widehat{u_{0}}(\xi)|^{2}d\xi+C(t+1)^{\gamma-1}\int_{|\xi|\leq g(t)}\frac{1}{|\xi|^{2(2\alpha-1)}}d\xi
\\ \leq& C(t+1)^{\gamma-1-\frac{3}{2\alpha}(\frac{2}{p}-1)}+ C(t+1)^{\gamma-1-\frac{5-4\alpha}{2\alpha}}.
\end{split}
\end{equation*}
Integrating in time from 0 to $t$ yields
\begin{equation}\label{add5}
\begin{split}
\|u(t)\|_{2}^{2}\leq C((t+1)^{-\gamma}+(t+1)^{-\frac{3}{2\alpha}(\frac{2}{p}-1)}+
(t+1)^{-\frac{5-4\alpha}{2\alpha}}).
\end{split}
\end{equation}

When $0<\alpha\leq\frac{1}{2}$ and $p\geq1\geq\frac{3}{4-2\alpha}$,
we have
$\frac{3}{2\alpha}(\frac{2}{p}-1)\leq\frac{5-4\alpha}{2\alpha}$.
Hence, by choosing $\gamma=\frac{3}{2\alpha}$, we obtain
\begin{equation*}
\|u(t)\|_{2}^{2}\leq C(t+1)^{-\frac{3}{2\alpha}(\frac{2}{p}-1)}.
\end{equation*}

When $\frac{1}{2}<\alpha\leq1$, two cases will be considered
respectively.  In case of $1\leq p<\frac{3}{4-2\alpha}$, one has
$\frac{3}{2\alpha}(\frac{2}{p}-1)>\frac{5-4\alpha}{2\alpha}.$ Hence,
by choosing $\gamma=3$, we have
\begin{equation}\label{add7}
\|u(t)\|_{2}^{2}\leq C(t+1)^{-\frac{5-4\alpha}{2\alpha}}.
\end{equation}
In case of $\frac{3}{4-2\alpha}\leq p<2$, one has
$\frac{3}{2\alpha}(\frac{2}{p}-1)\leq\frac{5-4\alpha}{2\alpha}$.
Hence, by choosing $\gamma=\frac{3}{2\alpha}$, we have
\begin{equation}\label{l1}
\|u(t)\|_{2}^{2}\leq C(t+1)^{-\frac{3}{2\alpha}(\frac{2}{p}-1)}.
\end{equation}

Now we  improve the decay rate in \eqref{add7}. We will use
 \eqref{add7} to show that
$$|\widehat{u}(\xi,t)|\leq |\widehat{u_{0}}(\xi)|+C ~\text{for}~
\xi\in S(t).$$  Then a bootstrap-type argument will lead to  a
better decay rate. Using \eqref{add6} and \eqref{add7}, for
$\frac{1}{2}<\alpha\leq1 ~\text{and}~\alpha\neq\frac{5}{6}$,  we
have
\begin{equation}\label{v1}
\begin{split}
\int_{0}^{t}e^{-|\xi|^{2\alpha}(t-s)}|H(\xi,s)|ds\leq & C|\xi|\int_{0}^{t}(s+1)^{-\frac{5-4\alpha}{2\alpha}}ds
\\ \leq& C\frac{2\alpha}{6\alpha-5}|\xi|((t+1)^{\frac{6\alpha-5}{2\alpha}}-1)
\\ \leq&
C\frac{2\alpha}{6\alpha-5}(t+1)^{-\frac{1}{2\alpha}}((t+1)^{\frac{6\alpha-5}{2\alpha}}-1)
\\ \leq&C.
\end{split}
\end{equation} If $\alpha=\frac{5}{6}$,  we have
\begin{equation}\label{v0}
\begin{split}
\int_{0}^{t}e^{-|\xi|^{2\alpha}(t-s)}|H(\xi,s)|ds\leq & C|\xi|\int_{0}^{t}(s+1)^{-1}ds
\\ \leq&
C(t+1)^{-\frac{3}{5}}\ln(t+1)
\\ \leq&C.
\end{split}
\end{equation}
Hence by \eqref{bd}, \eqref{v1} and \eqref{v0}
$$|\widehat{u}(\xi,t)|\leq |\widehat{u_{0}}(\xi)|+C, ~\text{for}~ \xi\in S(t).$$
This, combined with  \eqref{b5}, yields
\begin{equation*}
\begin{split}
\frac{d}{dt}((t+1)^{\gamma}\|u(t)\|^{2}_{2})\leq& C(t+1)^{\gamma-1}\int_{|\xi|\leq g(t)}|\widehat{u_{0}}(\xi)|^{2}d\xi+C(t+1)^{\gamma-1}\int_{|\xi|\leq g(t)}d\xi
\\ \leq& C(t+1)^{\gamma-1-\frac{3}{2\alpha}(\frac{2}{p}-1)}+ C(t+1)^{\gamma-1-\frac{3}{2\alpha}}.
\end{split}
\end{equation*}
Integrating with respect to time yields
\begin{equation*}
\begin{split}
\|u(t)\|_{2}^{2}\leq& C((t+1)^{-\gamma}+(t+1)^{-\frac{3}{2\alpha}(\frac{2}{p}-1)}+
(t+1)^{-\frac{3}{2\alpha}})\\ \leq&C((t+1)^{-\gamma}+(t+1)^{-\frac{3}{2\alpha}(\frac{2}{p}-1)}).
\end{split}
\end{equation*}
By choosing $\gamma$ suitably large, we have
\begin{equation*}
\|u(t)\|_{2}^{2}\leq
C(t+1)^{-\frac{3}{2\alpha}(\frac{2}{p}-1)}.
\end{equation*}

\end{proof}

\begin{proof}[\textbf{Proof of Theorem \ref{the1}}]
Two cases will be considered respectively.

{\it Case I.} When $1\leq\alpha<\frac{5}{4}$ and
$\frac{3}{4-2\alpha}\leq p<2$, similar to the proof of  \eqref{b5},
\eqref{add5} and \eqref{l1}, one has
\begin{equation*}
\|u(t)\|_{2}^{2}\leq C(t+1)^{-\frac{3}{2\alpha}(\frac{2}{p}-1)}.
\end{equation*}

{\it Case II.} When $1\leq\alpha<\frac{5}{4}$ and
$1\leq\frac{1}{3-2\alpha}\leq p<\frac{3}{4-2\alpha}$, similar to the
proof of \eqref{b5}, \eqref{add5} and \eqref{add7}, one has
\begin{equation}\label{l} \|u(t)\|_{2}^{2}\leq
C(t+1)^{-\frac{5-4\alpha}{2\alpha}}.
\end{equation}
It follows from \eqref{add6} and \eqref{l} that
\begin{equation}\label{add9}
\begin{split}
\int_{0}^{t}e^{-|\xi|^{2\alpha}(t-s)}|H(\xi,s)|ds\leq & C|\xi|\int_{0}^{t}(s+1)^{-\frac{5-4\alpha}{2\alpha}}ds
\\ \leq& C\frac{2\alpha}{6\alpha-5}|\xi|((t+1)^{\frac{6\alpha-5}{2\alpha}}-1)
\\ \leq&
C\frac{2\alpha}{6\alpha-5}(t+1)^{-\frac{1}{2\alpha}}((t+1)^{\frac{6\alpha-5}{2\alpha}}-1)
\\ \leq&C(t+1)^{\frac{3\alpha-3}{\alpha}}.
\end{split}
\end{equation}
Thanks to \eqref{bd}, we have $|\widehat{u}(\xi,t)|\leq
C(|\widehat{u_{0}}(\xi)|+(t+1)^{\frac{3\alpha-3}{\alpha}})$.
Applying \eqref{b5} again leads to
\begin{equation*}
\begin{split}
\frac{d}{dt}((t+1)^{\gamma}\|u(t)\|^{2}_{2})\leq& C(t+1)^{\gamma-1-\frac{3}{2\alpha}(\frac{2}{p}-1)}+ C(t+1)^{\gamma-1-\frac{3}{2\alpha}+\frac{6\alpha-6}{\alpha}}.
\end{split}
\end{equation*}
Integrating with respect to time and choosing $\gamma$ suitably
yield
\begin{equation*}
\|u(t)\|_{2}^{2}\leq
C(t+1)^{-\frac{3}{2\alpha}(\frac{2}{p}-1)}.
\end{equation*}
The proof of  Theorem \ref{the1} is finished.

\end{proof}

\begin{remark}
The proof of Theorems \ref{the0} and \ref{the1} is formal and we
assume that all the calculus   in the proof make sense. To make it
more rigorous, we apply the a prior estimates to the approximate
solutions constructed in the Appendix.  Let us recall that $u_{N}$
is a solution of the approximate equation
\begin{equation*}
\left\{\ba
&\partial _{t}u_{N}+PJ_{N}(u_{N}\cdot \nabla u_{N})+ \Lambda^{2\alpha}u_{N}=0,\\
&\Div~ u_{N}=0,\\
&u_{N}(x,0)=J_{N}u_{0},\\  \ea\ \right.
\end{equation*}
where $J_{N}$ is the spectral cutoff defined by $$\widehat{J_{N}f}(\xi)=1_{[0, N]}(|\xi|)\widehat{f}(\xi)$$ and $P$ is the Leray projector over divergence-free vector-fields.

\end{remark}

It is shown that the $u_{N}$ converges strongly in $L^{2}(0,T;
L^{2}_{loc}(\R^{3}))$ to a weak solution of  the generalized
three-dimensional Navier-Stokes equation \eqref{eq-1} in the
Appendix. Hence the $L^{2}$ decay of  $u_{N}$ will imply the $L^{2}$
decay of the weak solution of \eqref{eq-1}.

\section{Proof of Theorem \ref{the2} and Theorem \ref{addthe}}
\setcounter{section}{3}\setcounter{equation}{0}

In this section, we will give the proof of Theorem \ref{the2} and
Theorem \ref{addthe}. Before that,  we  recall the following result
established in \cite{[Wu]}.
\begin{theorem}\label{addd0}
Let $s\geq\frac{5}{2}-2\alpha$ with $0<\alpha<\frac{5}{4}$. Suppose
that $u_{0}\in H^{s}(\R^{3})$ with $\Div~ u_{0}=0$ and there exists
a constant $\epsilon$ such that $\|u_{0}\|_{H^{s}}\leq \epsilon$.
Then there exists a unique solution $u\in L^{\infty}(0,+\infty;
H^{s})$ satisfying
\begin{equation}\label{f}
\frac{d}{dt}\|u\|^{2}_{H^{s}}\leq -\|\Lambda^{\alpha}u\|^{2}_{H^{s}}.
\end{equation}
\end{theorem}

\begin{lemma}\label{Lem1}
Let $1\leq p<2$. Suppose that $u_{0}\in L^{p}(\R^3)\cap L^{2}(\R^3)$
with $\Div~ u_{0}=0$. Then, for any $|\xi|\leq1$ and $j\geq0$, we
have
\begin{equation}\label{e1}
|\widehat{\Lambda^{j}u}(\xi,t)|\leq C(|\widehat{u_{0}}(\xi)|+\frac{1}{|\xi|^{2\alpha-1}}), \end{equation} where $C$ depends only on $\|u_{0}\|_{L^{p}\cap L^{2}}$.
\end{lemma}

\begin{proof}
Since $|\xi|\leq1$, we have
\begin{equation*}
|\widehat{\Lambda^{j}u}(\xi,t)|\leq
|\xi|^{j}|\widehat{u}(\xi,t)|\leq |\widehat{u}(\xi,t)|.
\end{equation*}
Using Lemma \ref{lemb1} leads to the desired \eqref{e1}.
\end{proof}

The following are  decay estimates for high order derivatives of the
smooth solution.
\begin{theorem}\label{10-22-2}
Let $0<\alpha\leq 1$.  Suppose that $u_{0}\in L^{p}(\R^{3})\cap
H^{s}(\R^{3})$ with $1\leq p<2$ and  $s\geq\frac{5}{2}-2\alpha$,
satisfying $\Div~ u_{0}=0$. Then, there exists a $T_{0}>0$ such that
for any $t>T_{0}$ the global-in-time solution established in Theorem
\ref{addd0} satisfies
\begin{equation}\label{f1}
\|u(t)\|_{H^{s}}^{2}\leq C(t+1)^{-\frac{3}{2\alpha}(\frac{2}{p}-1)},
\end{equation}
 where  $C$ depends on $\alpha$ and $\|u_{0}\|_{H^{s}\cap L^{p}}$.
\end{theorem}

\begin{theorem}\label{addthe0}
Let $1\leq\alpha<\frac{5}{4}$.  Suppose that $u_{0}\in
L^{p}(\R^{3})\cap H^{s}(\R^{3})$ with  $s\geq\frac{5}{2}-2\alpha$
and $\frac{1}{3-2\alpha}\leq p<2$, satisfying $\Div~ u_{0}=0$. Then,
there exists a $T_{0}>0$ such that for any $t>T_{0}$ the
global-in-time solution established in Theorem \ref{addd0} satisfies
\begin{equation}\label{f1}
\|u(t)\|_{H^{s}}^{2}\leq C(t+1)^{-\frac{3}{2\alpha}(\frac{2}{p}-1)},
\end{equation}
 where  $C$ depends on $\alpha$ and $\|u_{0}\|_{H^{s}\cap L^{p}}$.
\end{theorem}

\begin{proof}[\textbf{Proof of Theorem \ref{10-22-2} and \ref{addthe0}}]
We adopt to the Fourier splitting method again. It follows from
\eqref{f} that
\begin{equation*}
\begin{split}
\frac{d}{dt}\int_{\R^{3}}|\widehat{u}(\xi,t)|^{2}+|\widehat{\Lambda^{s}u}(\xi,t)|^{2}d\xi\leq&-\int_{\R^{3}}|\xi|^{2\alpha} (|\widehat{u}(\xi,t)|^{2}+|\widehat{\Lambda^{s}u}(\xi,t)|^{2})d\xi.
\end{split}
\end{equation*}
Similar to the proof of Theorem \ref{the0} and Theorem \ref{the1},
we have
\begin{equation*}
\begin{split}
\frac{d}{dt}((t+1)^{\gamma}\int_{\R^{3}}|\widehat{u}(\xi,t)|^{2}+|\widehat{\Lambda^{s}u}(\xi,t)|^{2}d\xi)\leq& \gamma(t+1)^{\gamma-1}\int_{|\xi|\leq g(t)}|\widehat{u}(\xi,t)|^{2}+|\widehat{\Lambda^{s}u}(\xi,t)|^{2}d\xi.
\end{split}
\end{equation*}
Similar to \eqref{1} and \eqref{2}, using  Lemma \ref{Lem1}, we get
\begin{equation*}
\|u(t)\|_{H^{s}}^{2}\leq C(t+1)^{-\frac{3}{2\alpha}(\frac{2}{p}-1)}
\end{equation*}
 for any $t>T_{0}$. The proof of Theorem \ref{10-22-2} and
\ref{addthe0} are finished.
\end{proof}

To prove Theorem \ref{the2} and \ref{addthe}, we first present the
following commutator estimate.
\begin{lemma}\label{commutor}
Let $s>0$ and $1<p<\infty$. Then
\begin{equation}\label{com1}
\|\Lambda^{s}(fg)\|_{p}\leq C\|\Lambda^{s}f\|_{p_{1}}\|g\|_{p_{2}}
+\|\Lambda^{s}g\|_{q_{1}}\|f\|_{q_{2}},
\end{equation}
\begin{equation}\label{com2}
\|\Lambda^{s}(fg)-f\Lambda^{s}g\|_{p}\leq C\|\Lambda^{s}f\|_{p_{1}}\|g\|_{p_{2}}
+\|\nabla f\|_{q_{1}}\|\Lambda^{s-1}g\|_{q_{2}},
\end{equation}
where $\frac{1}{p}=\frac{1}{p_{1}}+\frac{1}{p_{2}}=\frac{1}{q_{1}}+\frac{1}{q_{2}}$.
\end{lemma}
The proof is referred to \cite{[Katocommutor]} and the details are omitted here .

Now we give the proof of Theorem \ref{the2} and \ref{addthe}.

\begin{proof}[\textbf{Proof of Theorems \ref{the2}}]
For any $m\in \mathbb{N}$, applying $\Lambda^{m}$ on both sides of
the first equation of \eqref{eq-1}, multiplying the resulting
equation by $\Lambda^{m}u$ and integrating by parts, we obtain
\begin{equation}\label{k}
\frac{d}{dt}\|\Lambda^{m}u\|_{2}^{2}+\|\Lambda^{m+\alpha}u\|_{2}^{2}=
-\int_{\R^{3}}\Lambda^{m}u\cdot \Lambda^{m}(u\cdot \nabla u)dx.
\end{equation}
By \eqref{com1}, we have
\begin{equation}\label{f3}
\begin{split}
\frac{d}{dt}\|\Lambda^{m}u\|_{2}^{2}+\|\Lambda^{m+\alpha}u\|_{2}^{2}\leq & C\|\Lambda^{m+\alpha}u\|_{2}\|\Lambda^{m-\alpha+1}(u\otimes u)\|_{2}\\ \leq &
\|\Lambda^{m+\alpha}u\|_{2}
\|\Lambda^{m-\alpha+1}u\|_{2}\|u\|_{\infty}.
\end{split}
\end{equation}
Since $$\|\Lambda^{m-\alpha+1}u\|_{2}\leq
C\|\Lambda^{m+\alpha}u\|_{2}^{\frac{1}{\alpha}-1}
\|\Lambda^{m}u\|_{2}^{2-\frac{1}{\alpha}},
~~\frac{1}{2}\leq\alpha\leq1,$$ one has
\begin{equation}\label{h1}
\begin{split}
\frac{d}{dt}\|\Lambda^{m}u\|_{2}^{2}+\|\Lambda^{m+\alpha}u\|_{2}^{2}
\leq& \|\Lambda^{m+\alpha}u\|_{2}^{\frac{1}{\alpha}}
\|\Lambda^{m}u\|_{2}^{2-\frac{1}{\alpha}}\|u\|_{\infty}\\ \leq &\frac{1}{4}\|\Lambda^{m+\alpha}u\|_{2}^{2}+C\|u\|_{\infty}^{\frac{2\alpha}{2\alpha-1}}
\|\Lambda^{m}u\|_{2}^{2}.
\end{split}
\end{equation}
Using Theorem \ref{the0} and Theorem \ref{10-22-2} yields
\begin{equation}\label{h}
\|u\|_{\infty}\leq C\|u\|_{2}^{\frac{1}{2}}\|\Lambda
^{3}u\|_{2}^{\frac{1}{2}}\leq C(t+1)^{-\frac{3}{4\alpha}}
\end{equation}
for   any $t>T_{0}$ and $\frac{1}{2}\leq\alpha\leq1$.
 Putting \eqref{h} into
\eqref{h1}, one has
\begin{equation}\label{h2}
\begin{split}
\frac{d}{dt}\|\Lambda^{m}u\|_{2}^{2}+\|\Lambda^{m+\alpha}u\|_{2}^{2}
\leq&C(t+1)^{-\frac{3}{4\alpha-2}}\|\Lambda^{m}u\|_{2}^{2}
\\ \leq&C(t+1)^{-1}\|\Lambda^{m}u\|_{2}^{2}
\end{split}
\end{equation}
for $\frac{1}{2}<\alpha\leq1$. In the case of
$0<\alpha\leq\frac{1}{2}$, we can also establish the similar
estimate as in \eqref{h2}. Indeed, by divergence free condition,
\eqref{k} can be rewritten as
\begin{equation}\label{k1}
\frac{d}{dt}\|\Lambda^{m}u\|_{2}^{2}+\|\Lambda^{m+\alpha}u\|_{2}^{2}=
-\int_{\R^{3}}\Lambda^{m}u\cdot (\Lambda^{m}(u\cdot\nabla u)-u\cdot\nabla\Lambda^{m} u)dx.
\end{equation}
Use the commutator estimate \eqref{com2} to get
\begin{equation}\label{k2}
\begin{split}
\frac{d}{dt}\|\Lambda^{m}u\|_{2}^{2}+\|\Lambda^{m+\alpha}u\|_{2}^{2}
\leq&C\|\nabla u\|_{\infty}\|\Lambda^{m}u\|_{2}^{2}.
\end{split}
\end{equation}
It follows from Theorem \ref{the0} and Theorem \ref{10-22-2} that,
for any $t>T_{0}$ and $0<\alpha\leq\frac{1}{2}$,
\begin{equation}\label{k3}
\begin{split}
\|\nabla u\|_{\infty}\leq& C\|u\|_{2}^{\frac{1}{6}}\|\Lambda
^{3}u\|_{2}^{\frac{5}{6}} \leq C(t+1)^{-\frac{3}{4\alpha}}
\\ \leq&C(t+1)^{-1}.
\end{split}
\end{equation}
Hence, we obtain that, for $0<\alpha\leq1$,
\begin{equation*}
\frac{d}{dt}\|\Lambda^{m}u\|_{2}^{2}+\|\Lambda^{m+\alpha}u\|_{2}^{2}
\leq C(t+1)^{-1}\|\Lambda^{m}u\|_{2}^{2}.
\end{equation*}
Let $$D_{i}(t)=\{\xi\in \R^{3}: |\xi|\leq f_{i}(t)\},~
l>\frac{3}{2\alpha}+\frac{m}{\alpha},~
f_{i}(t)=(\frac{l+i}{t+1})^{\frac{1}{2\alpha}}.~ i=0,1.$$
Then
\begin{equation}\label{c3}
\begin{split}
\|\Lambda^{m+\alpha}u\|_{2}^{2}=&\int_{\R^{3}}|\xi|^{2\alpha}|\widehat{\Lambda^{m}u}(\xi,t)|^{2}d\xi
\\ \geq&\int_{|\xi|\geq f_{i}(t)}|\xi|^{2\alpha}|\widehat{\Lambda^{m}u}(\xi,t)|^{2}d\xi
\\ \geq& f_{i}^{2\alpha}(t)\|\Lambda^{m}u\|_{2}^{2}-f_{i}^{2\alpha+2}(t)\int_{|\xi|\leq f_{i}(t)}|\widehat{\Lambda^{m-1}u}(\xi,t)|^{2}d\xi\\ \geq& f_{i}^{2\alpha}(t)\|\Lambda^{m}u\|_{2}^{2}-f_{i}^{2\alpha+2}(t)
\int_{\R^{3}}|\widehat{\Lambda^{m-1}u}(\xi,t)|^{2}d\xi,
\end{split}
\end{equation}
where $i=0,1$. Inserting \eqref{c3} with $i=1$ into  \eqref{h2}, we
get
\begin{equation}\label{c5}
\frac{d}{dt}\|\Lambda^{m}u\|_{2}^{2}+\frac{l}{t+1}\|\Lambda^{m}u\|_{2}^{2}
\leq C(\frac{l+1}{t+1})^{\frac{2\alpha+2}{2\alpha}}\|\Lambda^{m-1}u\|_{2}^{2}.
\end{equation}
To complete the proof, we use the inductions for $m$. The case $m=0$
has been proved in Theorem \ref{the0}. Assume that
$$\|\Lambda^{m-1}u\|_{2}^{2}\leq C_{m-1}(t+1)^{-\rho_{m-1}}, ~~
\rho_{m-1}=\frac{3}{2\alpha}+\frac{m-1}{\alpha}.$$ Then, thanks to
\eqref{c5}, we have
\begin{equation}\label{c6}
\frac{d}{dt}((t+1)^{l}\|\Lambda^{m}u\|_{2}^{2}) \leq
C_{m-1}(t+1)^{l-\rho_{m-1}-\frac{\alpha+1}{\alpha}}.
\end{equation}
Integrating  \eqref{c6} in time from  $T_{0}$ to $t$ yields
\begin{equation*}
(t+1)^{l}\|\Lambda^{m}u(t)\|_{2}^{2}
\leq(T_{0}+1)^{l}\|\Lambda^{m}u(T_{0})\|_{2}^{2}+C_{m-1}(t+1)^{l-\rho_{m-1}-\frac{1}{\alpha}},
\end{equation*}
which implies
\begin{equation*}
\begin{split}
\|\Lambda^{m}u(t)\|_{2}^{2}
\leq &C_{m}(t+1)^{-\rho_{m-1}-\frac{1}{\alpha}}
\\ \leq &C_{m}(t+1)^{-\frac{3}{2\alpha}-\frac{m}{\alpha}}.
\end{split}
\end{equation*}
The proof of Theorems \ref{the2} is finished.
\end{proof}

\begin{proof}[\textbf{Proof of Theorems \ref{addthe}}]
In  case of $1\leq\alpha<\frac{5}{4}$, since
$\frac{5}{2}-2\alpha\in(0, \frac{1}{2})$,  we obtain that, for any
$m\in \mathbb{N}$, $m\geq\frac{5}{2}-2\alpha$. Therefore,  Theorem
\ref{addd0} implies
\begin{equation}\label{llm}
\frac{d}{dt}\|\Lambda^{m}u\|^{2}_{2}\leq -\|\Lambda^{m+\alpha}u\|^{2}_{2}.
\end{equation}
Inserting \eqref{c3} with $i=0$ into \eqref{llm} yields
\begin{equation}\label{addd1}
\frac{d}{dt}\|\Lambda^{m}u\|_{2}^{2}+\frac{l}{t+1}\|\Lambda^{m}u\|_{2}^{2}
\leq C(\frac{l}{t+1})^{\frac{2\alpha+2}{2\alpha}}\|\Lambda^{m-1}u\|_{2}^{2}.
\end{equation}
Adopting to similar procedure in the proof of Theorem \ref{the2}, we
finish the proof of Theorem \ref{addthe}.
\end{proof}

\appendix
\section{ Existence of weak solutions}
\label{Appendix}
\setcounter{section}{4}\setcounter{equation}{0}
In this section we show that the generalized Navier-Stokes equations with $\alpha>0$
have a global weak solution corresponding to any prescribed $L^{2}$ initial data.

We start with a definition of weak solutions for \eqref{eq-1} with $L^{2}$ initial data $u_{0}$. Let $T>0$ be arbitrarily fixed.

\begin{definition}\label{app}
The function pair $(u(x,t), p(x,t))$ is called a weak solution of the problem \eqref{eq-1} if the following conditions are satisfied:
\begin{enumerate}[(1)]
\item
$u\in L^{\infty}(0,T;L^{2}(\R^{3}))\cap L^{2}(0,T;H^{\alpha}(\R^{3})),$
\item for any $\Phi\in C_{0}^{\infty}([0,T)\times\R^{3})$ with $\Phi(\cdot, T)=0$, we have $$\int_{0}^{T}\langle u, \Phi_{t} \rangle-\langle \Lambda^{\alpha}u, \Lambda^{\alpha}\Phi \rangle-\langle u\cdot \nabla u, \Phi\rangle dt=-\langle u(0), \Phi(0) \rangle,$$
\item $\Div~ u(x, t)=0$ for a.e. $(x, t)\in\R^{3}\times[0,T).$
\end{enumerate}
\end{definition}

The following theorem states that there exists  global-in-time weak solutions of \eqref{eq-1}.
\begin{theorem}\label{app0}
Let $T>0$ be fixed and $\alpha>0$. Assume that $u_{0}\in L^{2}(\R^{3})$. Then the system \eqref{eq-1} possess a weak solution obeying  Definition \ref{app} over $[0, T]$.
\end{theorem}

We will use the Friedrichs method to prove Theorem \eqref{app0}.
Before that, let us recall the following Picard theorem
\cite{[Majda]} and Bernstein inequality \cite{[1]}.
\begin{theorem}\label{app1}(Picard Theorem  on a Banach Space).
Let $O\subseteq B$ be an open subset of a  Banach Space $B$ and let $F:O\rightarrow B$ be a mapping that satisfies the following parameters:
\begin{enumerate}[(i)]
\item
$F(X)$ maps $O$ to $B$.
\item $F$ is locally Lipschitz continuous, i.e., for any $X\in O$ there exists $L>0$ and an open neighborhood $U_{X}$ of $X$ such that $$\|F(\widetilde{X})-F(\widehat{X})\|_{B}\leq L\|\widetilde{X}-\widehat{X}\|_{B}, ~~\text{for all}~~ \widetilde{X}, \widehat{X}\in U_{X}.$$
\end{enumerate}
Then, for any $X_{0}\in O$, there exists a time $T$ such that the ODE $$\frac{dX}{dt}=F(X), X|_{t=0}=X_{0}\in O$$ has a unique (local) solution $X\in C^{1}[(-T, T); O]$. In addition, the unique solution $X(t)$ either exists globally in time, or $T<\infty$ and $X(t)$ leaves the open set $O$ as $t\nearrow T$.
\end{theorem}

\begin{proposition}\label{App0}(Bernstein inequality).
Let B be a ball of $\R^{d}$. Then, there exists a positive
constant C such that for all integer $k\geq0$, all $b\geq a\geq 1$ and $u\in{L^{a}}$, the following
estimates are satisfied: $$\sup_{|\alpha|=k} \|\partial^{\alpha}u\|_{b}\leq C^{k+1}\lambda^{k+d(\frac{1}{a}-\frac{1}{b})}\|u\|_{a}, ~~ supp \hat{u}\subset \lambda B.$$
\end{proposition}

\begin{proof}[\textbf{Proof of Theorems \ref{app0}}]
For $N\geq1$, let $J_{N}$ be the spectral cutoff defined by $$\widehat{J_{N}f}(\xi)=1_{[0, N]}(|\xi|)\widehat{f}(\xi).$$
Let $P$ denote the Leray projector over divergence-free vector-fields. Consider the following ODE in the space $L_{N}^{2}=\{f\in L^{2}(\R^{3}):~supp \widehat{f}(\xi)\subset B(0, N)\},$
\be\label{app2}
\left\{\ba
&u_{t}+PJ_{N}(PJ_{N}u\cdot PJ_{N}\nabla u)+ PJ_{N}\Lambda^{2\alpha}u=0,\\
&u(x,0)=J_{N}u_{0}.\\  \ea\ \right. \ee We shall apply Picard
Theorem to show the existence (local) and uniqueness of solution to
\eqref{app2}. We write $$\frac{du}{dt}=-PJ_{N}(PJ_{N}u\cdot
PJ_{N}\nabla u)- PJ_{N}\Lambda^{2\alpha}u\triangleq F(u).$$ Then $F$
satisfies the local Lipschitz condition. In fact, for any $u, v \in
L_{N}^{2}$, by the H\"{o}lder inequality and the Bernstein
inequality,  we get
\begin{equation*}
\begin{split}
&\|PJ_{N}(PJ_{N}u\cdot PJ_{N}\nabla u)-PJ_{N}(PJ_{N}v\cdot PJ_{N}\nabla v)\|_{2}
\\ \leq&\|PJ_{N}(PJ_{N}(u-v)\cdot \nabla PJ_{N} u)\|_{2}+\|PJ_{N}(PJ_{N}v\cdot\nabla PJ_{N}(u-v))\|_{2}
\\ \leq& \|PJ_{N}(u-v)\|_{2}\|\nabla PJ_{N} u\|_{\infty}+\|\nabla PJ_{N}(u-v))\|_{2} \|PJ_{N} v\|_{\infty}
\\ \leq& N^{\frac{5}{2}}(\|u\|_{2}+\|v\|_{2})\|u-v\|_{2}.
\end{split}
\end{equation*}
By the Bernstein  inequality, it follows that
\begin{equation*}
\begin{split}
&\|PJ_{N}\Lambda^{2\alpha}u-PJ_{N}\Lambda^{2\alpha}v\|_{2}\\ \leq&
\|J_{N}\Lambda^{2\alpha}(u-v)\|_{2}\\ \leq& N^{2\alpha}\|u-v\|_{2}.
\end{split}
\end{equation*}
Consequently,
$$\|F(u)-F(v)\|_{2}\leq (N^{\frac{5}{2}}(\|u\|_{2}+\|v\|_{2})+N^{2\alpha})\|u-v\|_{2}.$$
Picard Theorem implies that \eqref{app2} has a unique local (in time) solution $u_{N}\in C^{1}([0, T_{N}); L^{2}_{N})$.
Recall that $P^{2}=P$, $J_{N}^{2}=J_{N}$ and $PJ_{N}=J_{N}P$, it is easy to check that $Pu_{N}$ and $J_{N}u_{N}$ are also solutions of \eqref{app2}. By the uniqueness, $Pu_{N}=u_{N}$ (i.e. $\Div u_{N}=0$) and $J_{N}u_{N}=u_{N}$. Then \eqref{app2} can be simplified as
\be\label{app3}
\left\{\ba
&\partial _{t}u_{N}+PJ_{N}(u_{N}\cdot \nabla u_{N})+ \Lambda^{2\alpha}u_{N}=0,\\
&\Div~ u_{N}=0,\\
&u_{N}(x,0)=J_{N}u_{0}.\\  \ea\ \right. \ee
Multiplying the first equation of\eqref{app3} by $u_{N}$ and integrating by parts, we obtain
$$\frac{1}{2}\frac{d}{dt}\|u_{N}(t)\|_{2}^{2}+\|\Lambda^{\alpha }u_{N}(t)\|_{2}^{2}=0,$$
which implies that
\be\label{app4}
 \|u_{N}(t)\|_{2}^{2}+2\int^{t}_{0}\|\Lambda^{\alpha }u_{N}(s)\|_{2}^{2}\,ds=\|u_{N}(0)\|_{2}^{2}\leq \|u_{0}\|_{2}^{2}.\ee
This implies that $u_{N}$ remains bounded in $L^{2}_{N}$ for finite time, whence $T_{N}=T$.

Next, we will use Aubin-Lions lemma \cite{[Temam]} to prove the strong convergence of $u_{N}$ (or its subsequence) in $L^{2}(0,T;L^{2}(\Omega))$ for any $\Omega\subset \R^{3}$. In fact, for any $h\in L^{2}(0,T; H^{3}(\R^{3}))$ and $\alpha\leq\frac{5}{2}$, we obtain
\be\label{60}\ba
\int^{T}_{0}\langle PJ_{N}(u_{N}(s)\cdot\nabla u_{N}(s)),\,h(s) \rangle \,ds&\leq
\int^{T}_{0}\|u_{N}(s)\|_{2}\|u_{N}(s)\|_{\frac{6}{3-2\alpha}}
\|\nabla h(s)\|_{\frac{3}{\alpha}}\,ds\\
 &\leq C\int^{T}_{0}\|u_{N}(s)\|_{2}\|\Lambda^{\alpha}u_{N}(s)\|_{2}
\|\nabla^{3} h(s)\|_{2}^{\frac{5-2\alpha}{6}}\|h(s)\|_{2}^{\frac{1+2\alpha}{6}}\,ds\\ &
\leq C\|u_{N}\|_{L^{\infty}(0,T;\, L^{2}(\R^{3}))}
 \|\Lambda^{\alpha}u_{N}\|_{L^{2}(0,T;\,L^{2}(\R^{3}))}
\|h\|_{L^{2}(0,\,T;\,H^{3}(\R^{3}))} \\
 &\leq C\|u_{0}\|_{2}^{2}\|h\|_{L^{2}(0,T;\,H^{3}(\R^{3}))},
\ea\ee where the H\"{o}lder inequality and the Gagliardo-Nirenberg
inequality have been used. The H\"older inequality and Sobolev
embedding $H^{3}(\R^{3})\hookrightarrow H^{\alpha}(\R^{3}),
~\alpha\leq\frac{5}{2}$ yield that
$$\ba
\int^{T}_{0}\langle\Lambda^{2\alpha}u_{N}(s), \,h(s) \rangle\, ds&\leq \int^{T}_{0}\|\Lambda^{\alpha}u_{N}(s)\|_{2}\|\Lambda^{\alpha}h(s) \|_{2}\,ds\\
&\leq \Big(\int^{T}_{0}\|\Lambda^{\alpha}u_{N}(s)\|_{2}^{2}\,ds\Big)^{1/2}
\Big(\int^{T}_{0}\|\Lambda^{\alpha}h(s)\|_{2}^{2}\,ds\Big)^{1/2}\\
&\leq \|u_{0}\|_{2}\| h\|_{L^{2}(0,T;\,H^{3}(\R^{3}))}.
\ea$$
Combining these estimates with the first equation of \eqref{app3}, we obtain
\be\label{App}
\partial_{t}u_{N}\in L^{2}(0,T;\,H^{-3}(\R^{3})),\ee
which together with \eqref{app4} yields that
$$u_{N}\rightarrow u ~~~\text{in} ~~~L^{2}(0,T;L^{2}(\Omega)) ~~\text{for any}~~ \Omega\subset \R^{3}.$$

We choose $\Omega_{1}\subset\Omega_{2}\subset\Omega_{3}\subset...$ with smooth boundary satisfying $\cup_{i=1}^{\infty} \Omega_{i}=\R^{3}$. For any fixed $i=1,2,...$, we obtain that there exists a subsequence of $\{u_{N}\}_{N=1}^{\infty}$
still denote by itself, such that $u_{N}$ strongly converges  $u$ in $L^{2}(0,T;L^{2}(\Omega_{i}))$. By the diagonal principle, there exists a subsequence $\{u_{N_{j}}\}_{j=1}^{\infty}$ of $\{u_{N}\}_{N=1}^{\infty}$ such that
$u_{N_{j}}$ strongly converges  $u$ in $L^{2}(0,T;L^{2}(\Omega_{i}))$ for any $i=1,2,...$ and hence in $L^{2}(0,T;L^{2}_{loc}(\R^{3}))$. These convergence guarantee that $u(x,t)$ is a weak solution of \eqref{eq-1}.

When $\alpha>\frac{5}{2}$, it can be proved in a similar way that
system \eqref{eq-1} possess a weak solution obeying  Definition
\ref{app}. The proof of the Theorem is finished.

\end{proof}


\end{document}